\documentclass[lettersize,journal]{IEEEtran}
\usepackage{amsmath}
\usepackage{amssymb,amsfonts}
\usepackage{amsthm}
\usepackage{algorithmic}
\usepackage{algorithm}
\usepackage{array}
\usepackage{boondox-cal}
\usepackage{textcomp}
\usepackage{stfloats}
\usepackage{url}
\usepackage{verbatim}
\usepackage{graphicx}
\usepackage{cite}
\usepackage{bm}
\usepackage{subfigure}
\usepackage{booktabs}
\usepackage[justification=centering]{caption}
\newtheorem{assumption}{Assumption}

\newtheorem{lemma}{Lemma}
\newtheorem{theorem}{Theorem}

\begin{document}

\title{Dynamic Regret for Online Composite Optimization}

\author{Ruijie Hou, Xiuxian Li, {\em Senior Member, IEEE}, and Yang Shi, {\em Fellow, IEEE}
	
\thanks{This research was supported by the National Natural Science Foundation of China under Grant 62003243, and the Shanghai Municipal Science and Technology Major Project, No. 2021SHZDZX0100.}
\thanks{R. Hou and X. Li are with Department of Control Science and Engineering, College of Electronics and Information Engineering, Tongji University, Shanghai 201800, China (e-mail: hourj21@tongji.edu.cn, xxli@ieee.org).}
\thanks{X. Li is also with the Shanghai Research Institute for Intelligent Autonomous Systems, Shanghai 201210, China.}
\thanks{Y. Shi is with Department of Mechanical Engineering, University of Victoria,
	Victoria, B.C. V8W 3P6, Canada (e-mail: yshi@uvic.ca).}
}

\markboth{January~2023}%
{Shell \MakeLowercase{\textit{et al.}}: A Sample Article Using IEEEtran.cls for IEEE Journals}


\maketitle

\begin{abstract}
This paper investigates online composite optimization in dynamic environments, where each objective or loss function contains a time-varying nondifferentiable regularizer. To resolve it, an online proximal gradient algorithm is studied for two distinct scenarios, including convex and strongly convex objectives without the smooth condition. In both scenarios, unlike most of works, an extended version of the conventional path variation is employed to bound the considered performance metric, i.e., dynamic regret. In the convex scenario, a bound $\mathcal{O}(\sqrt{T^{1-\beta}D_\beta(T)+T})$ is obtained which is comparable to the best-known result, where $D_\beta(T)$ is the extended path variation with $\beta\in[0,1)$ and $T$ being the total number of rounds. In strongly convex case, a bound $\mathcal{O}(\log T(1+T^{-\beta}D_\beta(T)))$ on the dynamic regret is established. In the end, numerical examples are presented to support the theoretical findings.
\end{abstract}


\section{Introduction}\label{s1}
Composite optimization has been studied in considerable research in recent decades. 
Usually, the problems are composed of functions with different properties.
This structure can be applied in statistics, machine learning, and engineering, such as empirical risk minimization, logistic regression, compressed sensing and so on\cite{9541509}.
There are many algorithms proposed for solving composite optimization, such as Second Order Primal-Dual Algorithm \cite{9547812}, ADA-FTRL and
ADA-MD\cite{JOULANI2020108}, COMID\cite{inproceedings}, ODCMD\cite{9070199}, SAGE\cite{Hu2009AcceleratedGM}, AC-SA\cite{Ghadimi2012OptimalSA}, RDA\cite{Chen2012OptimalRD}, most of which use the well-known proximal operator.
Proximal operator introduced by \cite{article3} is an efficient method to minimize functions with a convex but possibly nondifferentiable part. 

This paper focuses on online composite optimization in dynamic environments. 
It is one of numerous important settings in online optimization.
Among them, online convex optimization has received immense attentions among scientists with many applications in dynamic setting, such as, online trajectory optimization\cite{9165951}, network resource allocation\cite{allocation}, radio resource management\cite{9500493}, and so on. The typical setup of online optimization can be described below. At each time $t$, a learner selects an appropriate action $x_t \in \mathcal{X}$, where $\mathcal{X} \subset \mathbb{R}^n$ is a feasible constraint set and then an adversary selects a corresponding convex loss $F_t \in \mathcal {F}$, where $\mathcal {F}$ is a given set of loss functions available to the adversary. 
The objective of online optimization is selecting proper actions $\{x_t\}_{t=1}^T$ such that $\sum_{t=1}^TF_t(x_t)$ is minimized to track the best decision in hindsight. 

\textit{Static regret} is a classical performance metric: $\bm{Reg}_T^s:=\sum_{t=1}^TF_t(x_t)-\sum_{t=1}^TF_t(x^*)$, where $x^*$ minimizes the sum of all loss functions over the constraint set\cite{6698348}.
An effective algorithm ensures the average static regret goes to zero over an infinite time horizon\cite{7390084}. 
For instance, Online Gradient Method (OGM) is proved to have bounds $O(\sqrt{T})$ and $O(\log T)$ when $F_t$ is convex and strongly convex, respectively\cite{bubeck2011introduction}. 
By using Online Proximal Gradient (OPG) algorithm for a composite function, where one term is convex or strongly convex, the same bounds are also obtained\cite{9178808,Hu2009AcceleratedGM}. Although both of them share the same static regret bound, OPG can solve the composite loss function with nondifferentiable part, such as objective functions regularized by $\cal{l_1}$ norm. 
Besides, there are other related works solving online composite problems based on static regret.
Authors in \cite{9070199} developed an online distributed composite mirror descent using the Bregman divergence, and attained an average static regret bound of order $\mathcal{O}(\frac{1}{\sqrt{T}})$ for each agent, which they called average regularized regret. Even for multipoint bandit feedback model, they also achieved the same result.

In a wide range of applications, $F_t(x)$ often has different minimizers at each $t$. As such, the \textit{dynamic regret} introduced in \cite{Zinkevich2003} is a well-known extension of static regret.
Many previous works \cite{9184135,9525208} have pointed out that the dynamic regret is hard to achieve sublinearity with respect to $T$, since the function sequence fluctuates arbitrarily.
In order to bound the dynamic regret, one needs to specify some complexity measures.
By introducing $\{u_t\}_{t=1}^T$ as a sequence of any feasible comparators, several common complexity measures are showed as follows:
\subsubsection{The functions variation}
$$V(T):=\sum_{t=2}^Tsup_{u \in \mathcal {X}}
|F_t(u)-F_{t-1}(u)|.$$
$V(T)$ is used to solve problems where cost functions may change along a non-stationary variant of a sequential stochastic optimization. 
It has been obtained minimax regret bounds  $\mathcal{O}(T^{\frac{2}{3}}V(T)^{\frac{1}{3}})$ and $\mathcal{O}(T^{\frac{1}{2}}V(T)^{\frac{1}{2}})$ for convex and strongly convex functions, respectively, in a non-stationary setting under noisy gradient access\cite{2015Besbes}.

\subsubsection{The gradients variation}
$$F(T):=\sum_{t=2}^T\lVert \nabla F_t(u_t)-M_t\lVert_2^2,$$
where $\{M_1,M_2,...,M_T\}$ is computed accroding to the past observations or side information. 
For instance, one of the choice could be $M_t=\nabla F_{t-1}(u_{t-1})$\cite{9325943}.
This work about online Frank-Wolfe algorithm established $\mathcal{O}(\sqrt{ T}(1+V(T)+\sqrt{F(T)}))$ dynamic regret for convex and smooth loss function in non-stationary settings.
\subsubsection{The path variation}
\begin{align*}
	D(T):=\sum_{t=2}^T\Arrowvert u_t-u_{t-1}\Arrowvert,\\
	S(T):=\sum_{t=2}^T\Arrowvert u_t-u_{t-1}\Arrowvert^2.
\end{align*}
If the reference points move slowly, $S(T)$ may be smaller than $D(T)$. On the contrary, $D(T)$ may be smaller than $S(T)$. Specially, $D(T)$ and $S(T)$ are also utilized to represent those variations in the optimal points $x_t^*:=arg\min_{x\in \mathcal{X}}f_t(x)$\cite{2020arXiv200605876Z,10.5555/3294771.3294841}. Related works on path variation are significantly more than the first two.
Earlier research on OGM showed that the dynamic regret has an upper bound $\mathcal{O}(\sqrt{T}(D(T)+1))$ for convex cost functions\cite{Zinkevich2003}.
Then, authors improved that there exists a lower bound $\mathcal{O}(\sqrt{T(D(T)+1)})$\cite{Zhang2018AdaptiveOL}. 
By adding strict conditions, a regret bound of order $\mathcal{O}(D(T)+1)$ was established for strongly and smooth convex loss functions\cite{2016Mokhtari}.
Another study on Online Multiple Gradient Descent for smooth and strongly convex functions improved the regret bound to $\mathcal{O}(min(D(T), S(T)))$\cite{10.5555/3294771.3294841}.

\subsubsection{An extended path variation}
\begin{align}
	 D_\beta(T):=\sum_{t=2}^Tt^\beta\Arrowvert u_t-u_{t-1}\Arrowvert, \label{var}
\end{align}
where $0\leqslant\beta< 1$\cite{2021pgmDynamicRegret}.
Larger weights are allocated for the future parts than the previous parts. When $\beta=0$, $D_0(T)$ is equal to $D(T)$ and when $0<\beta<1$, $D_0(T)<D_\beta(T)<T^\beta D_0(T)$. 
Therefore, $D_\beta(T)$ is an extension of $D(T)$.
In the literature, online optimization based on $D_\beta(T)$ is few. 
One related work using $D_\beta(T)$ is on proximal online gradient algorithm\cite{2021pgmDynamicRegret}. 
The authors established a dynamic regret of order $\mathcal{O}(\sqrt{T^{1-\beta}D_\beta(T)}+\sqrt{T})$ for convex loss functions.

\begin{table*}
	\centering
	\caption{Summary of results on dynamic regret ($F_t=f_t+r_t$)}
	\begin{tabular}{cccc}
		\toprule
		\textbf{Reference}& \textbf{Algorithm} & \textbf{Assumptions on Loss functions}& \textbf{Dynamic regret}\\
		\midrule
		\cite{Zhang2018AdaptiveOL}& OGD& $F_t$ : Convex without $r_t$& $\mathcal{O}(\sqrt{ T(D_0(T)+1)})$\\
		\cite{2016Mokhtari}&OGD& $F_t$ : Strongly convex and smooth without $r_t$&$\mathcal{O}(D_0(T)+1)$\\
		\cite{2021pgmDynamicRegret}&POG&$f_t$ : Convex with fixed regularizer&$\mathcal{O}(\sqrt{T^{1-\beta}D_\beta(T)}+\sqrt{T})$\\
		\cite{Dixit2021OnlineLO}&DP-OGD&$f_t$ : Strongly convex and smooth &$\mathcal{O}(\log T(1+D_0(T)))$\\
		\cite{2020Ajalloeian}&OIPG&$f_t$ : Strongly convex and smooth&$\mathcal{O}(1+D_0(T))$\\
		This work&OPG&
		$f_t$ : Convex&$\mathcal{O}(\sqrt{T^{1-\beta}D_\beta(T)+T})$\\
		This work&OPG&$f_t$ : Strongly convex&$\mathcal{O}(\log T(1+T^{-\beta}D_\beta(T)))$\\
		\bottomrule
	\end{tabular}
\end{table*}

In addition, there are plenty of works studying about online composite optimization. 
In \cite{2020Ajalloeian}, inexact online proximal gradient algorithm was proposed for loss functions composed of strongly convex, smooth term and convex term, where the loss relies on the 
estimated gradient information of the smooth convex
function and the approximate proximal operator. 
Coincidentally, for allowing the loss functions to be subsampled, inexact gradient was added into online proximal gradient algorithm for a finite sum of composite functions\cite{RN74}.
When both algorithms are applied to the exact gradient case, the upper bounds are both in order of $\mathcal{O}(1+D(T))$. 
Furthermore, considering inexact scenarios by approximate gradients, approximate proximal operator and communication noise\cite{9654953}, online inexact distributed proximal gradient method (DPGM) guaranteed Q-linearly convergence to the optimal solution with some error. 
Decentralized algorithms are also proposed for solving composite optimization, as the composite loss function of each agent may be held privately. 
An unconstrained problem is considered firstly.
An online DPGM was proposed in \cite{Dixit2021OnlineLO} for each private cost function composed of strongly convex, smooth function and convex nonsmooth regularizer, obtaining a regret bound of order $\mathcal{O}(\log T(1+D(T)))$. 
Later, some works use mirror descent algorithm for constrained problems.
A distributed online primal-dual mirror descent algorithm is developed to handle composite objective functions, which consist of local convex functions and regularizers with time-varying coupled inequality constraints\cite{8950383}.

In this paper, we revisit OPG algorithm for composite functions with a time-varying regularizer and discuss dynamic regret in terms of $D_\beta(T)$. Most related works about OPG consider a time-invariant $r(x)$\cite{2021pgmDynamicRegret,10.5555/1577069.1755882,9287537,article7,9287550}, while our problem extends the stationary regularizer to time-varying scenario.

The contributions of this paper include:
(a) solving online composite optimization problems with two time-varying nonsmooth parts in terms of $D_\beta(T)$;
(b)	obtaining a best-known dynamic regret bound for composite functions with two convex but nondifferentiable functions;
(c) establishing a dynamic regret bound of OPG for functions composed of a strongly convex term and a convex but nondifferentiable term.

The paper is organized as follows. Section \ref{s2} describes the online composite optimization problem. 
In Section \ref{s3}, OPG algorithm is introduced, necessary assumptions are presented and main results are provided for two cases.
Numerical examples of the OPG algorithm are given in Section \ref{s4}. The conclusions are presented in Section \ref{s5}.

\textit{Notations:} Let $\lVert \cdot\rVert_1$ represent the $\cal{l}_1$ norm and $\lVert \cdot\rVert$ represent the $\cal{l}_2$ norm by default. 
$\partial f(x)$ denotes the subdifferential of a
convex function $f$ at $x$ and
$\tilde{\nabla} f(x)$ represents one subgradient in $\partial f(x)$.
Let $[N]:=\{1,2,...,N\}$.
Let $\langle \cdot, \cdot \rangle$ be the Euclidean inner product on $\mathbb{R}^n$. 

\section{Problem Formulation}\label{s2}
The loss function at each time $t\geq 0$ is composed of two parts as follows:
\begin{align}
	F_t(x)&:=f_{t}(x)+r_{t}(x),  ~~s.t.~x\in \cal {X}            \label{1}
\end{align}
where $\mathcal{X}\subset\mathbb{R}^n$ is a feasible constraint set, and functions $f_t,r_t:\mathcal{X}\to\mathbb{R}$ are not necessarily differentiable. 
At each instant $t$, an action $x_t \in \mathcal{X}$ is selected by the learner and then loss functions $f_{t }(x_t),r_{t }(x_t): \mathcal{X}\to \mathbb{R}$ are chosen by an adversary and revealed to the learner. The proper action sequence $\{x_t\}_{t=1}^T$ is expected to be obtained such that the dynamic regret is sublinear with the horizon $T$, i.e., $\lim_{T\to\infty}\frac{Reg_T^d}{T}=0$, where
\begin{align}
	\bm{Reg}_T^d:=\sum_{t=1}^TF_t(x_t)-\sum_{t=1}^TF_t(u_t),\label{2}
\end{align}
and $\{u_t\}_{t=1}^T$ is a sequence of any feasible comparators. It is worth noting that an important case of dynamic regret is when $u_t=x_t^*:=argmin_{x\in \mathcal{X}}F_t(x)$.
 

Examples of applications using the time-varying nonsmooth $f_t$ and the time-varying $r_t$ are briefly presented below.

\subsubsection{Examples of time-varying nonsmooth $f_t(x)$}
In machine learning, signal processing and statistics, many nonsmooth error functions have a wide range of applications for solving regression and classification problems\cite{Yang2014AnEP}.
In online streaming data-generating process, the probability density function may change with time\cite{7296710}.
There are some examples for online nonsmooth convex loss functions, by setting $y_t$ as label and $a_t$ as feature vector.
\begin{itemize}
	\item Hinge loss: 
	$$f_t(x)=max(0,1-y_ta_t^\top x).$$
	\item Generalized hinge loss:
	\begin{equation*}
		f_t(x)=\left\{
		\begin{aligned}
			&1-\alpha y_ta_t^\top x, ~~ &if~ y_ta_t^\top x\leqslant0 \\
			&1-y_ta_t^\top x, ~~ &if~ 0<y_ta_t^\top x<1 \\
			&0, ~~ &if~ y_ta_t^\top x\geqslant 1,
		\end{aligned}
		\right.
	\end{equation*}
	where $\alpha>1$.
	\item Absolute loss:
	$$
	f_t(x)=max_{\alpha\in[-1,1]}\alpha(y_t-a_t^\top x).
	$$
	\item $\epsilon$-insensitive loss:
	$$
	f_t(x)=max(|y_t-a_t^\top x|-\epsilon,0),
	$$
	where $\epsilon$ is a small postive constant.
\end{itemize}

\subsubsection{Examples of time-varying nondifferentiable $r_t(x)$}

\begin{itemize}
	\item The weighted $\cal{l}_1$ norm function:\\
	 Assume the optimal solution is sparse with many zero or near-zero coefficients. To solve it, let $r_t(x)$ be the weighted $\cal{l}_1$ norm function\cite{5495870,5947303,Yamamoto2012AdaptivePF}:
	\begin{align}\label{eq31}
		r_t(x):=\rho \sum_{i=1}^{N}\omega_i^{t}|x_i|,
	\end{align}
	where $x_i$ is the $i$-th component of $N$-dimensional $x$ and $\rho > 0$ is a parameter for regularizers. The weight $\omega_i^{t}$
	$(i = 1, 2,...,N)$ is defined by
	\begin{equation*}
		\omega_i^{t}:=\left\{
		\begin{aligned}
			\epsilon, &~~ if~ |x_i^{t-1}|>\tau \\
			1, &~~ otherwise,
		\end{aligned}
		\right.
	\end{equation*}
	with a parameter $\tau > 0$ and a small $\epsilon>0$, where $x_i^{t-1}$ is $t-1$-th iteration of $x_i$ by any proposed algorithm.
	The parameter $\tau$ may rely on noise statistics, e.g., the variance.
	By introducing the weights $\omega_i^{t}$, larger $|x_i|$ has smaller coefficients.
	Hence, larger $|x_i|$ renders a less sensitivity to the 
	impact of the proximity operator.
	Therefore, the weighted $\cal{l}_1$ norm exploits the sparseness more efficiently.
	\item The indicator function of time-varying set constraints:\\
	In a network flow problem, to handle the constraints $x \in \cal{X}_t$ with $\cal{X}_t \subset \cal{X}$ being a convex set, the authors in \cite{2020Ajalloeian} use indicator functions based on a given network $(\cal{N},\mathcal{E})$.
	The relevant variables are defined as follows.
	$z(i,s)$ denotes the rate produced at node $i$ for traffic $s$ and $x(ij,s)$ is the flow between nodes $i$ and $j$ for traffic $s$. For brevity, let $z$ and $x$ be the stacked traffic and link rates, respectively.
	Time-varying nonsmooth $r_t(z,x)$ is composed of two functions $r(z,x)$ and $\delta_{\cal{X}_t}(z,x)$, where $r(z,x)=\frac{\nu}{2}(\lVert z\lVert^2+\lVert x\lVert^2)$ is convex with some $\nu>0$ and $\delta_{\cal{X}_t}(z,x)$ is an indicator function of constraints, defined as
	\begin{equation*}
			\delta_{\cal{X}_t}(z,x):=\left\{
			\begin{aligned}
					+\infty , &~~ x\notin \cal{X}_t \\
					0, &~~ x\in \cal{X}_t.
				\end{aligned}
			\right.
		\end{equation*}
	Here are the relevant constraints in the netwrok. The capacity constraint for link is $0\leqslant \sum_{s}x(ij,s)+w_t(ij)\leqslant c_t(ij)$, where $c_t(ij)$ is the time-varying link capacity and $w_t(ij)$ is a time-varying and non-controllable link traffic. The flow conservation constraint implies $z(s) = \bar{R} x(s)$ with the routing matrix $\bar{R}$. $z^{max}$ represents the maximal traffic rate.
	Therefore, the time-varying nonsmooth $r_t(z,x)$ can be written as 
	\begin{align*}
		&r_t(z,x)=\sum_{i,s}\delta_{\{0\leqslant \sum_{s}x(ij,s)+w_t(ij)\leqslant c_t(ij)\}}\\
		&+\delta_{\{0\leqslant z\leqslant z^{max}\}}+\sum_{s}\delta_{\{z(s)=\bar{R}x(s)\}}+\frac{\nu}{2}(\lVert z\lVert^2+\lVert x\lVert^2).
	\end{align*}
\end{itemize}

\section{Main Results}\label{s3}
This section will introduce the OPG alogrithm and present the main results.

Firstly, how the OPG Algorithm updates the iterate $x_{t}$ is introduced. At each iteration $t$, based on a subgradient of $f_t(x)$ and a positive step size $\eta_t$, the next iteration $x_{t+1}$ is computed as
\begin{align}
	x_{t+1}=prox_{\eta_tr_t}(x_t-\eta_t\tilde{\nabla} f_t(x_t)),\label{algo}
\end{align}
where the proximal operator is defined by
\begin{equation}\label{pm}
	prox_{\eta_t r_t}(x):=arg\min_{u\in \cal{X}} \{r_t(u)+\frac{1}{2\eta_t}\lVert u-x\rVert^2\}.
\end{equation}
Then based on (\ref{algo}) and (\ref{pm}), the update is equivalent to
\begin{align}\label{x_t+1}
	x_{t+1 }=arg \min_{x\in \mathcal{X}}r_{t }(x)+\langle x, \tilde{\nabla} f_{t }(x_{t })\rangle+\frac{1}{2\eta_t}\rVert x-x_{t }\rVert^2.
\end{align}
At each iteration $t$, the algorithm assumes the availability of $\tilde{\nabla} f_{t }(x_{t })$.

\begin{algorithm}\
	\renewcommand{\algorithmicrequire}{$\textbf{Input:}$}
	\renewcommand{\algorithmicensure}{\textbf{Output:}}
	\caption{Online Proximal Gradient (OPG)}
	\label{alg1}
	\begin{algorithmic}[1]
		\REQUIRE Initial vector $x_1$, learning rate $\eta_t$ with $1\leqslant t\leqslant T$.
		\ENSURE $x_{2},x_{3},...,x_{T}$
		
		\FORALL{$t=1,2,...,T$}
	\STATE Compute the next action:$$	x_{t+1 }=arg \min_{x\in \mathcal{X}}\{r_{t }(x)+\langle x, \tilde{\nabla} f_{t }(x_{t })\rangle+\frac{1}{2\eta_t}\rVert x-x_{t }\rVert^2\}.$$
	\STATE \textbf{return} $x_{t+1}$
	\ENDFOR
\end{algorithmic}
\end{algorithm}

To proceed, the following standard assumptions are useful in the analysis.
\begin{assumption}\label{a1}
	Functions $f_{t }(x),r_{t }(x):\mathcal{X} \to \mathbb{R}$ are closed, convex for all $t\in [T]$.
\end{assumption}
Assumption \ref{a1} ensures the existence of optimal point for $f_{t}(x),r_{t}(x)$. Besides, according to convexity, for any $x,y\in \cal{X}$, any subgradient $\tilde{\nabla} f_{t}(x)\in \partial f_{t}(x)$ and $\tilde{\nabla} r_{t}(x)\in \partial r_{t}(x)$, one has:
\begin{align*}
	f_{t}(y)&\geqslant f_{t}(x)+\langle \tilde{\nabla} f_{t}(x), y-x\rangle,~ \forall x,y \in \mathcal{X}, \notag\\
	r_{t}(y)&\geqslant r_{t}(x)+\langle \tilde{\nabla} r_{t}(x), y-x\rangle,~ \forall x,y \in \mathcal{X}.
\end{align*}
\begin{assumption}\label{a2}
	Functions $f_{t }: \mathcal{X}\to \mathbb{R}$ is $\mu$-strongly convex for all $t\in [T]$. 
\end{assumption}
According to $\mu$-strongly convexity, for any 
$x,y \in \mathcal{X}$ and any subgradient $\tilde{\nabla} f_{t}(x)\in \partial f_{t}(x)$, one has:
$$f_{t}(y)\geqslant f_{t}(x)+\langle \tilde{\nabla} f_{t}(x), y-x\rangle + \frac{\mu}{2}\rVert x-y\rVert^2.$$
\begin{assumption}\label{a4}
	The set $\mathcal{X}$ is convex and compact, and $\rVert x-y\rVert\leq R$ for $R>0$ and all x,y $\in \mathcal {X}$.
\end{assumption}

\begin{assumption}\label{a5}
	For any $x\in\cal{X}$, $\rVert\tilde{\nabla} F_{t }(x)\rVert\leq M$ for some $M>0$.
\end{assumption}
If $\rVert\tilde{\nabla}f_{t}(x)\rVert$ and $\rVert\tilde{\nabla}r_{t}(x)\rVert$ respectively have upper bounds $m_f$ and $m_r$, one could set $M=m_f+m_r$.
\subsection{Dynamic Regret Bound for Convex $f_{t }$}\label{s3.1}
 In this case, the main result is provided in the following theorem.

\begin{theorem}\label{t0}
	Let Assumptions \ref{a1},\ref{a4},\ref{a5} hold. By setting non-increasing step size $\eta_{t}=t^{-\gamma}\sigma$ in the OPG, where $\sigma$ is a constant value
	$$\sigma=\frac{\sqrt{(1-\gamma)2RT^{2\gamma-\beta-1}D_\beta(T)+R^2T^{2\gamma-1}}}{M},$$
	and $\gamma\in [\beta,1)$, there holds
	\begin{align}\label{eq20}
		\bm{Reg}_T^d=\mathcal{O}(\sqrt{T^{1-\beta}D_\beta(T)+T}),
	\end{align}
where $D_\beta(T)$ is defined in (\ref{var}).
\end{theorem}
\begin{proof}
		See Appendix \ref{ap0}.
\end{proof}

\textit{Remark 1:}
Note that the bound (\ref{eq20}) is comparable to the best-known one $\sqrt{T^{1-\beta}D_\beta(T)}+\sqrt{T}$ established in \cite{2021pgmDynamicRegret}, where time-invariant regularizers are considered. However, more general time-varying regularizers are investigated here.

\textit{Remark 2:}
Note that the bound $\mathcal{O}(\sqrt{ T(D_0(T)+1)})$ in \cite{Zhang2018AdaptiveOL} is an optimal dynamic regret matching with the lower bound given in
\cite[Theorem 2]{Zhang2018AdaptiveOL}. 
By setting $\beta=0$, this bound is a special case of (\ref{eq20}).
%

\subsection{Dynamic Regret Bound for Strongly Convex $f_{t}$}\label{s3.2}
The following Lemma \ref{l1} gives an upper bound for any point in $\cal{X}$ and positive step size. 
Besides, when setting $\mu=0$, it also holds for the case without strongly convexity, which is discussed in Section \ref{s3.1}.
It will be useful for deriving the main results.
\begin{lemma}\label{l1}
	Let Assumptions \ref{a1},\ref{a2},\ref{a4},\ref{a5} hold. For any $u_t \in \cal{X}$, positive step size $\eta_t>0$ and iteration $x_t$ in the OPG, there holds
	\begin{align}
		F_{t}(x_{t})-F_{t}(u_{t}) &\leqslant\frac{1}{2}(\frac{1}{\eta_t}-\mu)\rVert u_{t}-x_t\rVert^2-\frac{1}{2\eta_t}\rVert u_{t}-x_{t+1}\rVert^2\notag\\
		&+\frac{\eta_t}{2}
		(\rVert \tilde{\nabla} f_{t }(x_t)+\tilde{\nabla} r_{t }(x_t)\rVert^2).\label{eq5}
	\end{align}
\end{lemma}
\begin{proof}
	See Appendix \ref{ap1}.
\end{proof}

Now by summing (\ref{eq5}) at each $t$, selecting appropriate step size $\eta_t$, the following theorem is obtained for characterizing the dynamic regret bound for strongly convex $f_t$.

\begin{theorem}\label{t1}
	Let Assumptions \ref{a1},\ref{a2},\ref{a4},\ref{a5} hold. For any  $0\leqslant\beta<1$ in OPG, setting the non-increasing step size $\eta_t=\frac{\gamma}{ t}$ and choosing $\gamma$ as
	\begin{align*}
		\gamma=\frac{2RT^{-\beta} D_\beta(T)+R^2}{\delta R^2+(\mu-\delta)\frac{1}{T}\lVert u_1-x_{1}\rVert^2},
	\end{align*}
where $\delta\in(0,\mu)$ is small enough such that $\gamma\delta<1$, and $u_{1}$ and $x_{1}$ are the initial comparator vector and the initial vector in OPG, respectively, there holds
	\begin{align}
		\bm{Reg}_T^d=\mathcal{O}(\log T(1+T^{-\beta}D_\beta(T))),
	\end{align}
where $D_\beta(T)$ is defined in (\ref{var}).
\end{theorem}
\begin{proof}
	See Appendix \ref{ap2}.
\end{proof}
It is worth noting that $\gamma$ in Theorem \ref{t1} is a fixed value but $\gamma$ in Theorem \ref{t0} is any positive value in $[\beta,1)$. They are differently chosen in different scenarios.


\textit{Remark 3:}
	If choosing $\beta=0$, dynamic regret is in order of $\mathcal{O}(\log (T)(1+D_0(T)))$, which is the result for distributed proximal gradient algorithm over dynamic graphs \cite{Dixit2021OnlineLO}. 
	However, Theorem \ref{t1} drops the assumption of $l$-smooth on $f_t(x)$ which is required in \cite{Dixit2021OnlineLO}. The existing works for composite optimization establishing upper bounds of dynamic regret under strongly convex $f_t(x)$ usually assume smoothness of $f_t(x)$ \cite{2020Ajalloeian,9654953,Alghunaim2021DecentralizedPG}. 
	Therefore, the result without the smooth assumption here greatly expands the scope of application.
	
\textit{Remark 4:}
	Theorem \ref{t1} also holds true when considering static regret.
	If the sequence $\{u_t\}_{t=1}^T$ is chosen time-invariant, which means the sequence does not change over time, $D_\beta(T)=\sum_{t=2}^Tt^\beta\Arrowvert u_t-u_{t-1}\Arrowvert$ is equal to 0. In this case, $\bm{Reg}_T^d= \mathcal{O}(\log(T))$ is obtained, which is the optimal bound on static regret in the strongly convex case\cite{9178808}. 
	Therefore, Theorem \ref{t1} may be optimal in this sense which is of interest as one possible future work.
	
\textit{Remark 5:}
	A promising extension can be considered in the future work.
	The positive step size $\eta_t$ is related to $D_\beta(T)$ in Theorems \ref{t0} and \ref{t1}. However, in most dynamic environment, $D_\beta(T)$ may be unknown in advance. In \cite{2021pgmDynamicRegret}, authors pointed out that in some online learning cases, $D_\beta (T)$ can be estimated by employing observed data for the learning problems. By adjusting step size, the dynamic regret is expected to decrese in an appropriate rate and then $D_\beta(T)$ may be derived in reverse.
	Moreover, it can be observed that the selection of step size in Theorems \ref{t0} and \ref{t1} needs to know $T$ in advance. However, if $T$ is unknown, similar results can still be ensured by leveraging the so-called doubling trick, as done in\cite{Yu2020ALC}.

\textit{Remark 6:}
The result in Theorem 2 can be further improved when $f_t$ is $\ell$-smooth as well.
	The result in \cite{2020Ajalloeian} shows that 	
	\begin{align}
		\lVert x_{t+1}-x_{t}^*\lVert\leqslant \rho_t\lVert x_{t}-x_{t}^*\lVert,
	\end{align}
	where $\rho_t:=max\{|1-\eta_t\mu|,|1-\eta_t \ell|\}$ and $x_{t}^*$ is the optimal point of $F_t(x)=f_t(x)+r_t(x)$. Based on the convexity of $F_t$ and the bound of $\rVert\tilde{\nabla}F_t\rVert$, the upper bound for strongly convex and smooth case is in order of $\mathcal{O}(1+D_0(T))$.
	The result matches the best-known one in \cite{2020Ajalloeian}.
	
\section{Numerical Examples}\label{s4}
Numerical experiments are provided here to show the performance of OPG by comparison with some state-of-the-art algorithms, which are valid for online composite convex optimizition.

\textbf{Objective functions:}
The performance of OPG is tested by optimizing two types of loss functions, corresponding to the function in the theoretical part.
Based on the nonsmooth convex Hinge loss and the time-varying nondifferentiable regularization function $r_t(x)$ defined by (\ref{eq31}), two online regression functions are defined:
\begin{align}
	F^1_t(x)&= l_t(x) + r_t(x),\label{f1}\\
	F_t^2(x)&=l_t(x)+\frac{\lambda}{2}\lVert x\rVert^2+ r_t(x).\label{f3}
\end{align}
Here $l_t(x)$ is the Hinge loss, defined by
\begin{align*}
	l_t(x)=max(0,1-y_ta_t^\top x),
\end{align*}
where $y_t$ and $a_t$ are generated by the dataset Usenet2. Usenet2 is based on the newsgroups collection and collects which messages are considered interesting or junk in each time period, where the 99-dimensional vector is realized as feature $a_t$ and the 1-dimensional vector is realized as label $y_t$. 
The related parameter of $r_t(x)$ is $\rho=0.4$, $\tau=1$ and $\epsilon=0.1$ for $F^1_t(x)$ and $F^2_t(x)$.
The related parameter of $f_t(x)=l_t(x)+\frac{\lambda}{2}\lVert x\rVert^2$ is $\lambda=1$ for $F_t^2(x)$.

The two functions are considered as different types of loss functions. 
Since the first term $f_t(x)$ of $F^1_t(x)$ is a nonsmooth convex loss $l_t(x)$,
$F^1_t(x)$ will show the perpformance of OPG for nonsmooth and convex $f_t(x)$. Since the first term $f_t(x)$ of $F^2_t(x)$ is a nonsmooth strongly convex loss $l_t(x)+\frac{\lambda}{2}\lVert x\rVert^2$, $F^2_t(x)$ will show the perpformance of OPG for nonsmooth and strongly convex $f_t(x)$. Besides, both of the loss functions have a time-varying nonsmooth regularization function.
Based on the results in theoretical analysis in Section \ref{s3}, OPG algorithm is able to solve both of them. 

\textbf{Other online methods:}
In order to show the effectiveness of the proposed OPG algorithm, let us consider the following efficient online algorithms.
\subsubsection{SAGE (Stochastic Accelerated GradiEnt)} It is an accelerated gradient method for stochastic composite optimization\cite{Hu2009AcceleratedGM}. 
Although it is a stochastic composite optimization, they also propose the accelerated gradient method for online setting.
Each iteration of the method takes the form
\begin{align*}
	x_t&=(1-\alpha_t)y_{t-1}+\alpha_tz_{t-1},\\
	y_t&=arg\min_{x} \{\langle \tilde{\nabla} f_{t-1}(x_t), x-x_t\rangle+\frac{L_t}{2}\lVert x-x_t\rVert^2 + r_t(x)\},\\
	z_t&=z_{t-1}-\alpha_t(L_t+\mu\alpha_t)^{-1}[L_t(x_t-y_t)+\mu(z_{t-1}-x_t)],
\end{align*}
where $\{\alpha_t\}$ and $\{L_t\}$ are two parameter sequences and $\mu$ is the strong-convexity constant of $F_t(x)$.
It has established upper bounds on static regret in order of ${\cal{O}}(\sqrt{T})$ for convex $f_{t}(x)$ and ${\cal{O}}(\log(T))$ for strongly convex $f_{t}(x)$, by choosing suitable parameter sequences $\{\alpha_t\}$ and $\{L_t\}$\cite{Hu2009AcceleratedGM}.
\subsubsection{AC-SA (Accelerated Stochastic Approximation)} It is a stochastic gradient
descent algorithm for handling convex composite optimization problems\cite{Ghadimi2012OptimalSA}. 
For better comparison, we choose the true subgradient $\tilde{\nabla} f_{t}$ instead of the stochastic one.
Each iteration of the method takes the form
\begin{align*}
	x_t^{md}=&\frac{(1-\alpha_t)(\mu+\gamma_t)}{\gamma_t+(1-\alpha_t^2)\mu}x_{t-1}^{ag}+\frac{\alpha_t[(1-\alpha_t)+\gamma_t]}{\gamma_t+(1-\alpha_t^2)\mu}x_{t-1},\\
	x_t=&arg\min_{x}\{\alpha_t[\langle \tilde{\nabla} f_{t}(x_t^{md}),x\rangle+r_t(x)+\mu V(x_t^{md},x)]\\
	&+[(1-\alpha_t)\mu+\gamma_t]V(x_{t-1},x)\},\\
	x_t^{ag}=&\alpha_t x_t +(1-\alpha_t)x_{t-1}^{ag},
\end{align*}
where $\{\alpha_t\}$ and $\{\gamma_t\}$ are two parameter sequences, $\mu$ denotes the strong-convexity constant of $F_t(x)$ and $V(\cdot,\cdot)$ is the Bregman divergence, and here we choose $V(x,z)=\frac{1}{2}\lVert x-z\rVert^2$.
Optimal convergence rates for tackling different types of stochastic composite optimization problems were discussed in \cite{Ghadimi2012OptimalSA}.
\subsubsection{RDA (Regularized Dual Averaging Method)} 
It is an extension of Nesterov’s dual averaging method for composite optimization problems, where one term is the convex loss function, and the other is a regularization term\cite{Chen2012OptimalRD}.
Each iteration of the method takes the form
\begin{align*}
	x_{t+1}=arg\min_{x}\{\frac{1}{t}\sum_{\tau = 1}^{t}\langle \tilde{\nabla} f_\tau{}(x_\tau), x\rangle + r_t(x)+ \frac{\beta_t}{t}h(x)\},
\end{align*}
where $\{\beta_t\}$ is a parameter sequence and $h(x)$ is an auxiliary strongly convex function, and here we choose $h(x)=\frac{1}{2}\lVert x\rVert^2$. It ensures a bound ${\cal{O}}(\sqrt{T})$ on static regret with $\beta_t={\cal{O}}(\sqrt{T})$ for convex loss\cite{Chen2012OptimalRD}.
 

\subsection{Performance on Nonsmooth Convex $f_t(x)$}
According to $F^1_t(x)$ in (\ref{f1}), a time-varying nondifferentiable regularization function is added into a nonsmooth convex loss function. 
The step size $\eta_t$ is set to $0.001/t$ in OPG, while step sizes in other three algorithms are set optimally based on their analysis.
The simulation is run 1500 times and the average of dynamic regret $\bm{Reg}_T^d/T$ is shown in Fig. \ref{fa}. It can be seen that OPG performs the best in our simulation. 
\begin{figure}[htbp]
	\centering
	\includegraphics[width=3in]{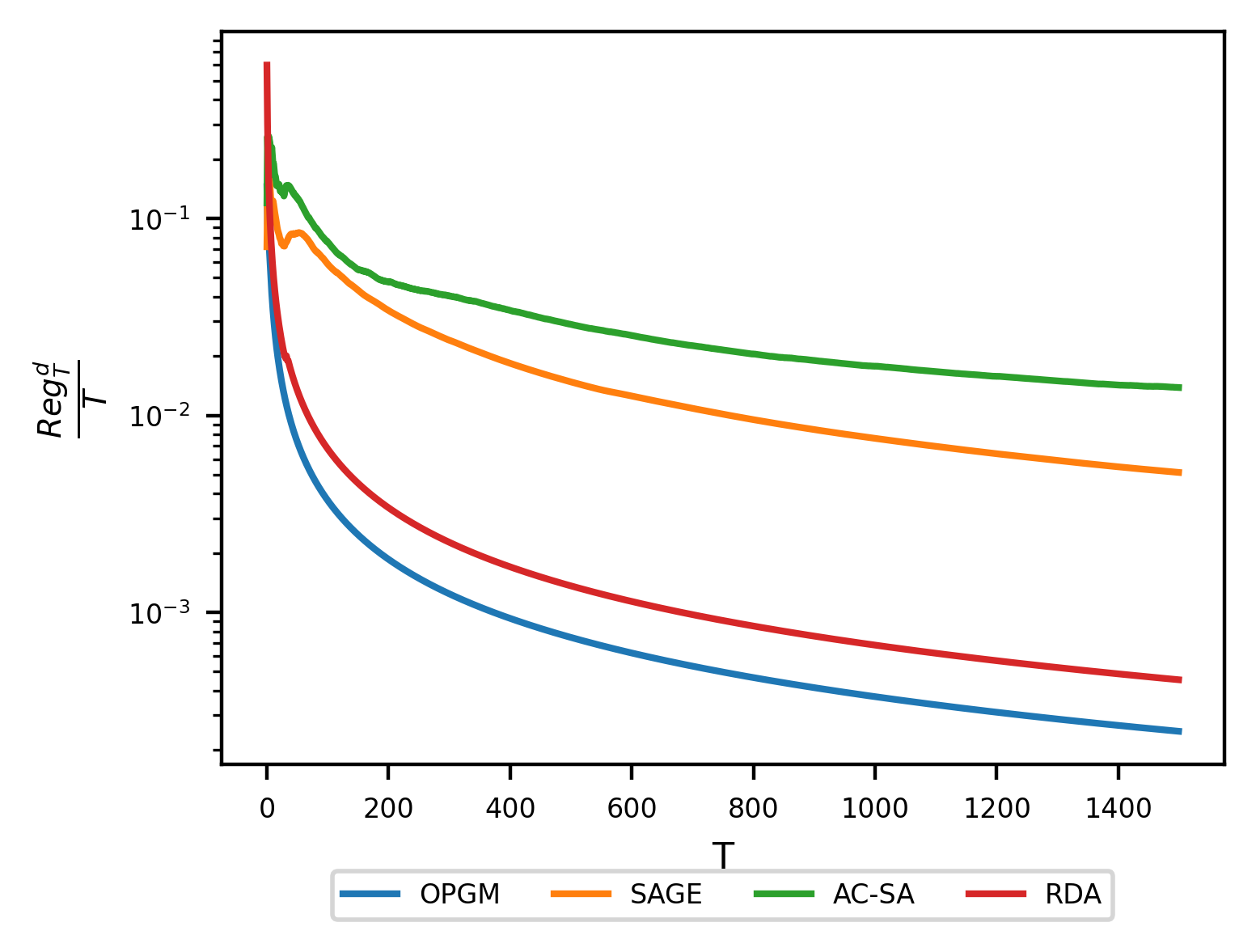}
	\caption{The performance on $F^1_t(x)$}\label{fa}
\end{figure}

\subsection{Performance on Nonsmooth Strongly Convex $f_t(x)$}
The example (\ref{f3}) will show the performance of OPG on nonsmooth strongly convex $f_t(x)$ and time-varying nondifferentiable $r_t(x)$. 
To achieve the best possible upper bound for each algorithm, different step sizes are set according to their analysis. The simulation is run 1500 times.
Fig. \ref{fb} shows how the average of dynamic regret $\bm{Reg}_T^d/T$ changes with horizon $T$. It shows that OPG performs the best in the simulation.

\begin{figure}[htbp]
	\centering
	\includegraphics[width=3in]{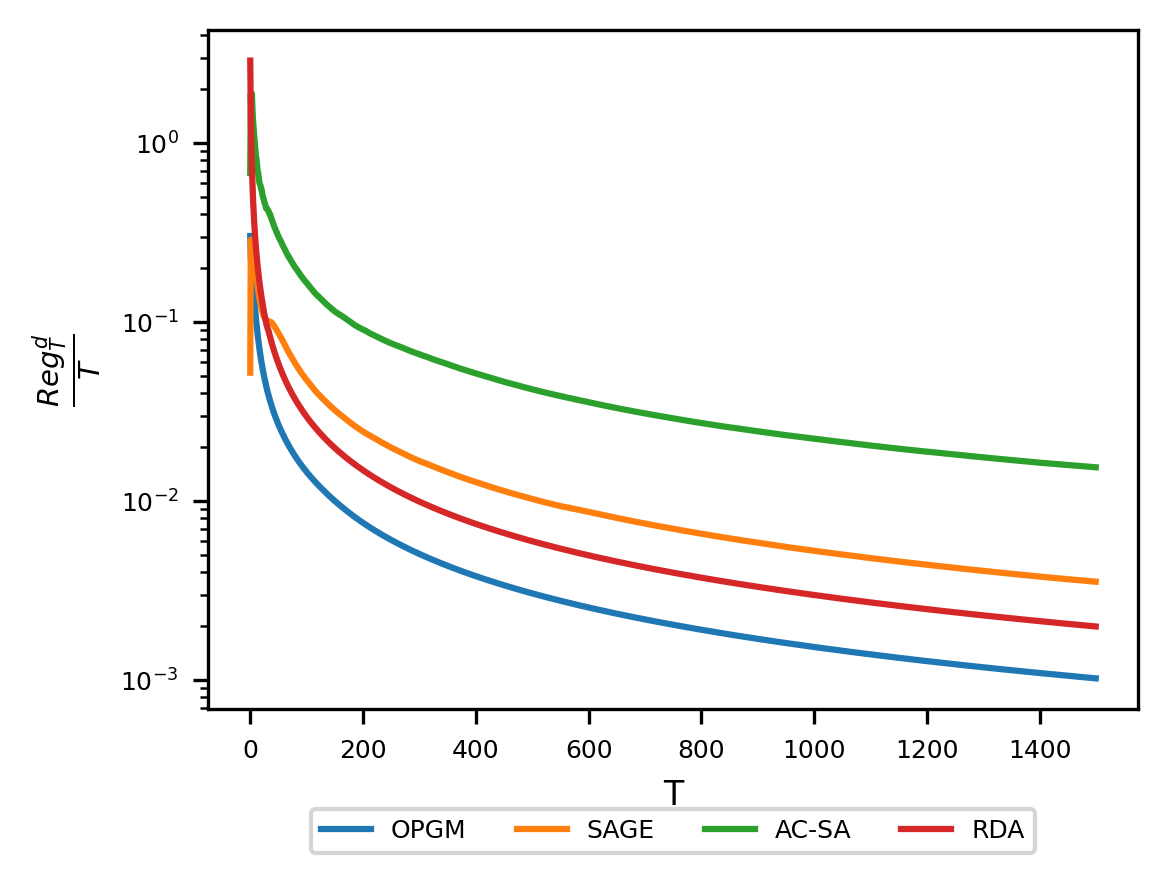}
	\caption{The performance on $F^2_t(x)$}\label{fb}
\end{figure}

Therefore, numerical results show that, compared with other three algorithms, OPG algorithm performs the best for the two cases in theoretical analysis. 

\section{Conclusion}\label{s5}
This paper considers online optimization with composite loss functions, where two terms are time-varying and nonsmooth.
The loss function is optimized by online proximal gradient (OPG) algorithm in terms of the extended path variation $D_\beta(T)$.
By analyzing upper bounds of its dynamic regret for two cases, their corresponding upper bounds are obtained.
In the first case that $f_t(x)$ is convex, we established a regret bound comparable to the best-known one in the existing literature, simultaneously extending the problem to time-varying regularizers.
In the second case that $f_t(x)$ is strongly convex, we obtained a better bound than another work while removing the smoothness assumption. 
At last, numerical experiments verified the performance of OPG algorithm in two cases. 
Possible future directions can be placed on the investigation of whether the obtained regret bounds are optimal or not, and the relaxation of stepsize selection without dependence on $D_\beta(T)$, etc.

{\appendix}

{\appendices
In this section, detailed discussions of the lemmas and theorems are given as follows.
\section{Proof of Theorem \ref{t0}}\label{ap0}
	First, the following result reveals the relationship between any sequence $\{u_t\}_{t=1}^T \subset \cal{X}$ and iteration sequence $\{x_t\}_{t=1}^T$ in OPG. If Assumption \ref{a4} is satisfied, for any non-increasing step size $0<\eta_{t+1}\leqslant\eta_{t}$ in OPG algorithm, then it follows from \cite[Lemma 4]{2021pgmDynamicRegret} that
\begin{align}
	&\sum_{t=1}^T\frac{1}{\eta_t}(\lVert u_t-x_{t}\rVert^2-\lVert u_t-x_{t+1}\rVert^2)\nonumber\\
	\leqslant&2R\sum_{t=1}^{T-1}\frac{1}{\eta_{t}}\lVert u_{t+1}-u_{t}\rVert+\frac{R^2}{\eta_{T}}, \label{e10}
\end{align}
where $R$ is defined in Assumption \ref{a4}.

When setting $\mu=0$, Lemma \ref{l1} amounts to the case of convex $f_t(x)$. Then, (\ref{eq5}) reduces to
\begin{align}\label{l1.1}
	F_{t}(x_{t})-F_{t}(u_{t}) &\leqslant\frac{1}{2\eta_t}\rVert u_{t}-x_t\rVert^2-\frac{1}{2\eta_t}\rVert u_{t}-x_{t+1}\rVert^2\notag\\
	&+\frac{\eta_t}{2}
	(\rVert \tilde{\nabla} f_{t }(x_t)+\tilde{\nabla} r_{t }(x_t)\rVert^2).	
\end{align}

According to (\ref{e10}), the sum of $\frac{1}{2\eta_t}\rVert u_{t}-x_t\rVert^2
-\frac{1}{2\eta_t}\rVert u_{t}-x_{t+1}\rVert^2$ could be bounded by $R\sum_{t=1}^{T-1}\frac{1}{\eta_{t}}\lVert u_{t+1}-u_{t}\rVert+\frac{R^2}{2\eta_{T}}$. 
Then Assumption \ref{a5} shows that the upper bound of $\rVert \tilde{\nabla} f_{t}(x_t)+\tilde{\nabla} r_{t}(x_t)\rVert^2$ is $M^2$.
Hence, (\ref{l1.1}) implies
\begin{align}
	&\sum_{t=1}^T(F_t(x_t)-F_t(u_t)) \nonumber\\
	\leqslant& R\sum_{t=1}^{T-1}\frac{1}{\eta_{t}}\lVert u_{t+1}-u_{t}\rVert+\frac{R^2}{2\eta_{T}}+\frac{M^2}{2}\sum_{t=1}^{T}\eta_{t}.\label{eq32}
\end{align}

If one multiplies and divides $t^\beta$ by $\lVert u_{t+1}-u_{t}\rVert$ and bounds $\frac{1}{\eta_tt^\beta}$ by choosing $\mathop{\max}_{t\in[T]}\big\{\frac{1}{\eta_tt^\beta}\big\}$, then according to the definition of $D_\beta(T)$, (\ref{eq32}) yields
\begin{align}
	&\sum_{t=1}^T(F_t(x_t)-F_t(u_t)) \nonumber\\
	\leqslant& R\mathop{\max}_{t\in[T]}\big\{\frac{1}{\eta_tt^\beta}\big\}\sum_{t=1}^{T-1}t^\beta\lVert u_{t+1}-u_{t}\rVert
	+\frac{R^2}{2\eta_{T}}\nonumber\\
	&+\frac{M^2}{2}\sum_{t=1}^{T}\eta_{t}\nonumber\\
	=& R\mathop{\max}_{t\in[T]}\big\{\frac{1}{\eta_tt^\beta}\big\}D_\beta(T)
	+\frac{R^2}{2\eta_{T}}+\frac{M^2}{2}\sum_{t=1}^{T}\eta_{t}.\label{e11}
\end{align}

By setting $\eta_{t}=t^{-\gamma}\sigma$ with $\sigma$ being a constant and $\gamma\in[\beta,1)$, $\mathop{\max}_{t\in[T]}\big\{\frac{1}{\eta_tt^\beta}\big\}=\frac{T^{\gamma-\beta}}{\sigma}$, invoking
(\ref{e11}) yeilds that 
\begin{align}
	&\sum_{t=1}^T(F_t(x_t)-F_t(u_t)) \nonumber\\
	\leqslant&\frac{RT^{\gamma-\beta}}{\sigma}D_\beta(T)+\frac{R^2T^\gamma}{2\sigma}+\frac{M^2\sigma}{2}\sum_{t=1}^{T}t^{-\gamma}\nonumber\\
	\leqslant&\frac{RT^{\gamma-\beta}}{\sigma}D_\beta(T)+\frac{R^2T^\gamma}{2\sigma}+\frac{M^2\sigma T^{1-\gamma}}{2(1-\gamma)}.
\end{align}
According to Cauchy inequality, if one chooses the optimal $\sigma$ with
$$
\sigma=\frac{\sqrt{(1-\gamma)2RT^{2\gamma-\beta-1}D_\beta(T)+R^2T^{2\gamma-1}}}{M},
$$
it can be then obtained that
\begin{align}
	&\sum_{t=1}^T(F_t(x_t)-F_t(u_t)) \nonumber\\
	\leqslant&\sqrt{\frac{M^2 (2RT^{1-\beta}D_\beta(T)+TR^2)}{1-\gamma}}. \label{eq.}
\end{align}
Thus, there holds
\begin{align*}
	\bm{Reg}_T^d=\mathcal{O}(\sqrt{T^{1-\beta}D_\beta(T)+T}).
\end{align*}
\hfill\rule{2mm}{2mm}

\section{Proof of Lemma \ref{l1}}\label{ap1}
Let us start with defining a Bregman divergence $B_\psi(x,u)$ associated with a differentiable strongly convex $\psi(x)$. It will be useful to introduce some results of Bregman divergence\cite{2021pgmDynamicRegret}:
\begin{itemize}
	\item[$a$.] $B_\psi(x,u):=\psi(x)-\psi(u)-\langle \nabla\psi(u), x-u\rangle$ (its definition);
	\item[$b$.] $\nabla_xB_\psi(x,u)=\nabla\psi(x)-\nabla\psi(u)$;
	\item[$c$.] $B_\psi(x,u)=B_\psi(x,z)+B_\psi(z,u)+\langle x-z, \nabla\psi(z)-\nabla\psi(u)\rangle$.
\end{itemize}

Let $\psi(x)=\frac{1}{2}\|x\|^2$ in the following. By Algorithm 1, setting any sequence $\{u_t\}_{t=1}^T \subset \cal{X}$, one has
\begin{align*}
	&F_t(x_t)-F_t(u_t)\nonumber\\
	=&f_t(x_t)+r_t(x_t)-f_t(u_t)-r_t(u_t).
\end{align*}

Invoking the optimality criterion for (\ref{x_t+1}) implies that for any $x\in \mathcal{X}$,
\begin{align}
	0\leqslant \langle x-x_{t+1}, \tilde{\nabla} f_{t}(x_{t})+\tilde{\nabla} r_{t}(x_{t+1})+\frac{1}{\eta_t}(x_{t+1}-x_{t})\rangle.
\end{align}
From this result, choosing $x=u_{t}$, it follows that
\begin{align}
	&\langle \tilde{\nabla} f_{t}(x_{t})+\tilde{\nabla} r_{t}(x_{t+1}), x_{t+1}-u_{t}\rangle \nonumber\\
	\leqslant &\frac{1}{\eta_t}\langle x_{t+1}-x_{t}, u_{t}-x_{t+1}\rangle. \label{eq2}
\end{align}
According to (\ref{eq2}), $f_t$'s strong convexity and $r_t$'s convexity imply that
\begin{align}
	&f_{t}(x_{t})+r_{t}(x_{t})-f_{t}(u_{t})-r_{t}(u_{t})\nonumber\\
	=&f_{t}(x_{t})+r_{t}(x_{t+1})-f_{t}(u_{t})-r_{t}(u_{t})\nonumber\\
	&+r_{t}(x_{t})-r_{t}(x_{t+1})\nonumber\\
	\leqslant& \langle \tilde{\nabla} f_{t}(x_{t}), x_{t}-u_{t}\rangle-\frac{\mu}{2}\lVert u_{t}-x_{t}\rVert^2\nonumber\\
	&+\langle \tilde{\nabla} r_{t}(x_{t+1}), x_{t+1}-u_{t}\rangle+\langle \tilde{\nabla} r_{t}(x_{t}), x_{t}-x_{t+1}\rangle\nonumber\\
	=&\langle \tilde{\nabla} f_{t}(x_{t}), x_{t+1}-u_{t}\rangle+\langle \tilde{\nabla} f_{t}(x_{t}), x_{t}-x_{t+1}\rangle\nonumber\\
	&-\frac{\mu}{2}\lVert u_{t}-x_{t}\rVert^2+\langle \tilde{\nabla} r_{t}(x_{t+1}), x_{t+1}-u_{t}\rangle\nonumber\\
	&+\langle \tilde{\nabla} r_{t}(x_{t}), x_{t}-x_{t+1}\rangle\nonumber	\\
	=&\langle \tilde{\nabla} f_{t}(x_{t})+\tilde{\nabla} r_{t}(x_{t+1}), x_{t+1}-u_{t}\rangle-\frac{\mu}{2}\lVert u_{t}-x_{t}\rVert^2\nonumber\\
	&+\langle \tilde{\nabla} f_{t}(x_{t})+\tilde{\nabla} r_{t}(x_{t}), x_{t}-x_{t+1}\rangle. \label{eq1}	
\end{align}

Combining (\ref{eq2}) with (\ref{eq1}) yields
\begin{align}
	&f_{t}(x_{t})+r_{t}(x_{t})-f_{t}(u_{t})-r_{t}(u_{t})\nonumber\\
	\leqslant&\frac{1}{\eta_t}\langle x_{t+1 }-x_{t }, u_{t }-x_{t+1 }\rangle -\frac{\mu}{2}\lVert u_{t }-x_{t }\rVert^2\nonumber\\
	&+\langle \tilde{\nabla} f_t(x_{t })+\tilde{\nabla} r_{t }(x_{t }), x_{t }-x_{t+1 }\rangle\nonumber\\
	=&\frac{1}{\eta_t}\langle\nabla\psi(x_{t+1 })-\nabla\psi(x_{t }), u_{t }-x_{t+1 } \rangle-\frac{\mu}{2}\lVert u_{t }-x_{t }\rVert^2 \nonumber\\
	&+\langle \tilde{\nabla} f_{t }(x_{t })+\tilde{\nabla} r_{t }(x_{t }), x_{t }-x_{t+1 }\rangle\nonumber\\
	=&\frac{1}{\eta_t}( B_\psi(u_{t },x_{t })-B_\psi(u_{t },x_{t+1 })- B_\psi(x_{t+1 },x_{t }))\nonumber\\
	&-\frac{\mu}{2}\lVert u_{t }-x_{t }\rVert^2+\langle \tilde{\nabla} f_{t }(x_{t })+\tilde{\nabla} r_{t }(x_{t }), x_{t }-x_{t+1 }\rangle. \label{el15}
\end{align}
According to the definition of $\psi(x)$, the term $\langle x_{t+1  }-x_{t  }, u_{t  }-x_{t+1  }\rangle$ is equal to $\langle\nabla\psi(x_{t+1  })-\nabla\psi(x_{t  }), u_{t  }-x_{t+1  } \rangle$. The last equation in (\ref{el15}) is established due to the property $c$ of Bregman divergence. Thus, by (\ref{el15}), invoking the definition of Bregman divergence leads to
\begin{align}
	&f_{t  }(x_{t  })+r_{t  }(x_{t  })-f_{t  }(u_{t  })-r_{t  }(u_{t  })\nonumber\\
	\leqslant&\frac{1}{2\eta_t}(\lVert u_{t}-x_{t}\rVert^2- \lVert u_{t}-x_{t+1}\rVert^2- \lVert x_{t+1}-x_{t}\rVert^2) \nonumber\\
	&-\frac{\mu}{2}\lVert u_{t}-x_{t}\rVert^2+\langle \tilde{\nabla} f_{t}(x_{t})+\tilde{\nabla} r_t(x_{t}), x_{t}-x_{t+1}\rangle.\label{eq4}
\end{align}
By Young's inequality, one can obtain that
\begin{align}
	&\langle \tilde{\nabla} f_{t}(x_{t})+\tilde{\nabla} r_{t}(x_{t}), x_{t}-x_{t+1}\rangle-\frac{1}{2\eta_t}\lVert x_{t}-x_{t+1}\rVert^2\nonumber\\
	\leqslant& \frac{\eta_t}{2}\lVert \tilde{\nabla} f_{t}(x_{t})+\tilde{\nabla} r_{t}(x_{t})\rVert^2. \label{eq3}
\end{align}
Thus, invoking (\ref{eq3}), together with (\ref{eq4}), yields that
\begin{align*}
&f_{t  }(x_{t  })+r_{t  }(x_{t  })-f_{t  }(u_{t  })-r_{t  }(u_{t  })\nonumber\\
	\leqslant&\frac{1}{2\eta_t}(\lVert u_{t}-x_{t}\rVert^2-\lVert u_{t}-x_{t+1}\rVert^2) \nonumber\\
	&-\frac{\mu}{2}\lVert u_{t}-x_{t}\rVert^2+\frac{\eta_t}{2}\lVert \tilde{\nabla} f_{t}(x_{t})+\tilde{\nabla} r_{t}(x_{t})\rVert^2.
\end{align*}
\hfill\rule{2mm}{2mm}
\section{Proof of Theorem \ref{t1}}\label{ap2}
According to Assumption \ref{a5},
it implies that $\rVert \tilde{\nabla} f_t(x)+\tilde{\nabla} r_t(x)\rVert^2\leq M^2$.
Applying this substitution into Lemma \ref{l1} yields
\begin{align}
	&\sum_{t=1}^T(F_t(x_t)-F_t(u_t))\nonumber\\
	=&\sum_{t=1}^T(f_t(x_t)+r_t(x_t)-f_t(u_t)-r_t(u_t))\nonumber\\
	\leqslant&\frac{1}{2}\sum_{t=1}^T((\frac{1}{\eta_t}-\mu)\lVert u_t-x_{t}\rVert^2-\frac{1}{\eta_t}\lVert u_t-x_{t+1}\rVert^2)\nonumber\\
	&+\frac{M^2}{2}\sum_{t=1}^T\eta_t.\label{eq7}
\end{align}

Now let us bound the first term of (\ref{eq7}) as follows
\begin{align}
	&\sum_{t=1}^T((\frac{1}{\eta_t}-\mu)\lVert u_t-x_{t}\rVert^2-\frac{1}{\eta_t}\lVert u_t-x_{t+1}\rVert^2)\nonumber\\
	=&\sum_{t=0}^{T-1}(\frac{1}{\eta_{t+1}}-\mu)\lVert u_{t+1}-x_{t+1}\rVert^2-\sum_{t=1}^T\frac{1}{\eta_t}\lVert u_t-x_{t+1}\rVert^2\nonumber\\
	=&\sum_{t=0}^{T-1}(\frac{1}{\eta_{t+1}}-\frac{1}{\eta_{t}}-\mu)\lVert u_{t+1}-x_{t+1}\rVert^2-\sum_{t=1}^T\frac{1}{\eta_t}\lVert u_t-x_{t+1}\rVert^2\nonumber\\
	&+\sum_{t=0}^{T-1}\frac{1}{\eta_{t}}\lVert u_{t+1}-x_{t+1}\rVert^2\nonumber\\
	=&\sum_{t=1}^{T-1}(\frac{1}{\eta_{t+1}}-\frac{1}{\eta_{t}}-\mu)\lVert u_{t+1}-x_{t+1}\rVert^2-\sum_{t=1}^T\frac{1}{\eta_t}\lVert u_t-x_{t+1}\rVert^2\nonumber\\
	&+\sum_{t=1}^{T-1}\frac{1}{\eta_{t}}\lVert u_{t+1}-x_{t+1}\rVert^2+(\frac{1}{\eta_1}-\mu)\lVert u_1-x_{1}\rVert^2\nonumber\\
	=&\sum_{t=1}^{T-1}(\frac{1}{\eta_{t+1}}-\frac{1}{\eta_{t}}-\mu)\lVert u_{t+1}-x_{t+1}\rVert^2+(\frac{1}{\eta_1}-\mu)\lVert u_1-x_{1}\rVert^2\nonumber\\
	&+\sum_{t=1}^{T-1}\frac{1}{\eta_{t}}(\lVert u_{t+1}-x_{t+1}\rVert^2-\lVert u_t-x_{t+1}\rVert^2). \label{eq6}
\end{align}
According to
\begin{align}
	\lVert u_t-x_{t+1}\rVert^2=&
	\lVert u_{t+1}-x_{t+1}\rVert^2-2\langle u_{t+1}-x_{t+1}, u_{t+1}-u_{t}\rangle\nonumber\\
	 &+\lVert u_{t+1}-u_{t}\rVert^2, \label{eq35}
\end{align}
and using (\ref{eq35}), one has
\begin{align}
	&\sum_{t=1}^T((\frac{1}{\eta_t}-\mu)\lVert u_t-x_{t}\rVert^2-\frac{1}{\eta_t}\lVert u_t-x_{t+1}\rVert^2)\nonumber\\
	\leqslant&\sum_{t=1}^{T-1}(\frac{1}{\eta_{t+1}}-\frac{1}{\eta_{t}}-\mu)\lVert u_{t+1}-x_{t+1}\rVert^2+(\frac{1}{\eta_1}-\mu)\lVert u_1-x_{1}\rVert^2\nonumber\\
	&+\sum_{t=1}^{T-1}\frac{2}{\eta_{t}}\langle u_{t+1}-x_{t+1}, u_{t+1}-u_{t}\rangle-\sum_{t=1}^{T}\frac{1}{\eta_{t}}\lVert u_{t+1}-u_{t}\rVert^2\nonumber\\
	\leqslant&\sum_{t=1}^{T-1}(\frac{1}{\eta_{t+1}}-\frac{1}{\eta_{t}}-\mu)\lVert u_{t+1}-x_{t+1}\rVert^2+(\frac{1}{\eta_1}-\mu)\lVert u_1-x_{1}\rVert^2\nonumber\\
	&+\sum_{t=1}^{T-1}\frac{2}{\eta_{t}}\lVert u_{t+1}-x_{t+1}\lVert \lVert u_{t+1}-u_{t}\lVert-\sum_{t=1}^{T}\frac{1}{\eta_{t}}\lVert u_{t+1}-u_{t}\rVert^2\nonumber\\
	\leqslant&\sum_{t=1}^{T-1}(\frac{1}{\eta_{t+1}}-\frac{1}{\eta_{t}}-\mu)\lVert u_{t+1}-x_{t+1}\rVert^2+(\frac{1}{\eta_1}-\mu)\lVert u_1-x_{1}\rVert^2\nonumber\\
	&+\sum_{t=1}^{T-1}\frac{2R}{\eta_{t}}\lVert u_{t+1}-u_{t}\lVert.\label{le6}
\end{align}
Considering this upper bound, the dynamic regret is bounded above by 
\begin{align}
	&\sum_{t=1}^T(f_t(x_t)+r_t(x_t)-f_t(u_t)-r_t(u_t))\nonumber\\
	\leqslant&\frac{M^2}{2}\sum_{t=1}^T\eta_t+\frac{1}{2}\sum_{t=1}^{T-1}(\frac{1}{\eta_{t+1}}-\frac{1}{\eta_{t}}-\mu)\lVert u_{t+1}-x_{t+1}\rVert^2\nonumber\\
	&+\frac{1}{2}(\frac{1}{\eta_1}-\mu)\lVert u_1-x_{1}\rVert^2+\sum_{t=1}^{T-1}\frac{R}{\eta_{t}}\lVert u_{t+1}-u_{t}\rVert\nonumber\\
	\leqslant&\frac{M^2}{2}\sum_{t=1}^T\eta_t+\frac{1}{2}\sum_{t=1}^{T-1}(\frac{1}{\eta_{t+1}}-\frac{1}{\eta_{t}}-\mu)\lVert u_{t+1}-x_{t+1}\rVert^2\nonumber\\
	&+\frac{1}{2}(\frac{1}{\eta_1}-\mu)\lVert u_1-x_{1}\rVert^2\nonumber\\
	&+R\mathop{\max}_{t\in[T]}\big\{\frac{1}{\eta_tt^\beta}\big\}\sum_{t=1}^{T-1}t^\beta\lVert u_{t+1}-u_{t}\rVert.\label{el16}
\end{align}

By introducing a small positive number $\delta<\mu$, (\ref{el16}) yeilds
\begin{align*}
		&\sum_{t=1}^T(f_t(x_t)+r_t(x_t)-f_t(u_t)-r_t(u_t))\nonumber\\
		\leqslant&\frac{M^2}{2}\sum_{t=1}^T\eta_t+\frac{1}{2}\sum_{t=1}^{T-1}(\frac{1}{\eta_{t+1}}-\frac{1}{\eta_{t}}-\delta)\lVert u_{t+1}-x_{t+1}\rVert^2\nonumber\\
		&+\frac{1}{2}(\frac{1}{\eta_1}-\delta)\lVert u_1-x_{1}\rVert^2-\frac{1}{2}(\mu-\delta)\lVert u_1-x_{1}\rVert^2\\
		&+R\mathop{\max}_{t\in[T]}\big\{\frac{1}{\eta_tt^\beta}\big\}\sum_{t=1}^{T-1}t^\beta\lVert u_{t+1}-u_{t}\rVert.
\end{align*}

Setting $\eta_t=\frac{\gamma}{t}$, where $\gamma<\frac{1}{\delta}$, one can derive $\frac{1}{\eta_{t+1}}-\frac{1}{\eta_{t}}-\delta$ and $\frac{1}{\eta_{1}}-\delta$ are both greater than 0. Thus, there holds
\begin{align}
	&\sum_{t=1}^T(f_t(x_t)+r_t(x_t)-f_t(u_t)-r_t(u_t))\nonumber\\
	\leqslant&\frac{M^2}{2}\sum_{t=1}^T\eta_t+\frac{R^2}{2}\sum_{t=1}^{T-1}(\frac{1}{\eta_{t+1}}-\frac{1}{\eta_{t}}-\delta)+\frac{R^2}{2}(\frac{1}{\eta_1}-\delta)\nonumber\\
	&-\frac{1}{2}(\mu-\delta)\lVert u_1-x_{1}\rVert^2+R\mathop{\max}_{t\in[T]}\big\{\frac{1}{\eta_tt^\beta}\big\}D_\beta(T)\nonumber\\
	\leqslant&\frac{M^2}{2}\sum_{t=1}^T\eta_t+R\mathop{\max}_{t\in[T]}\big\{\frac{1}{\eta_tt^\beta}\big\}D_\beta(T)-\frac{\delta R^2 T}{2}+\frac{R^2}{2\eta_T}\nonumber\\
	&-\frac{1}{2}(\mu-\delta)\lVert u_1-x_{1}\rVert^2,\label{le8}
\end{align}
where one uses Assumption \ref{a4} to bound $\lVert u_{t+1}-x_{t+1}\rVert^2$, $\lVert u_1-x_{1}\rVert^2$ and uses
the definition of $D_\beta(T)$. Then the result (\ref{le8}) is obtained by telescoping subtraction of $\sum_{t=1}^{T-1}(\frac{1}{\eta_{t+1}}-\frac{1}{\eta_{t}})$.

As $\eta_t=\frac{\gamma}{t}$ and $0\leqslant\beta<1$, $\mathop{\max}_{t\in[T]}\big\{\frac{1}{\eta_tt^\beta}\big\}$ can be bounded as
\begin{align}
	\mathop{\max}_{t\in[T]}\big\{\frac{1}{\eta_tt^\beta}\big\}=\mathop{\max}_{t\in[T]}\big\{\frac{t^{1-\beta}}{\gamma}\big\}=\frac{T^{1-\beta}}{\gamma}.\label{eq8}
\end{align}
According to the fact that the area of the integral sum from $1$ to $T$ is larger than the area of the discrete sum from $2$ to $T$, one can bound $\sum_{t=1}^T\eta_t$ as follows
\begin{align}
	\sum_{t=1}^T\eta_t&=\gamma\sum_{t=1}^T\frac{1}{t}=\gamma(1+\sum_{t=2}^T\frac{1}{t})\nonumber\\
	&\leqslant \gamma(1+\int_1^T\frac{1}{t}dt)=\gamma(1+\log T).\label{eq9}
\end{align}

Now, adding (\ref{eq8}) and (\ref{eq9}) into (\ref{le8}), it is obtained
\begin{align}
	&\sum_{t=1}^T(f_t(x_t)+r_t(x_t)-f_t(u_t)-r_t(u_t))\nonumber\\
	\leqslant&\frac{M^2\gamma}{2}(1+\log T)+\frac{ R T^{1-\beta}}{\gamma}D_\beta(T)-\frac{\delta R^2 T}{2}+\frac{ TR^2}{2\gamma}\nonumber\\
	&-\frac{1}{2}(\mu-\delta)\lVert u_1-x_{1}\rVert^2.\label{le9}
\end{align}

By choosing
\begin{align*}
	\gamma=\frac{2RT^{-\beta} D_\beta(T)+R^2}{\delta R^2+(\mu-\delta)\frac{1}{T}\lVert u_1-x_{1}\rVert^2},
\end{align*}
it is easy to verify that the term $\frac{ R T^{1-\beta}}{\gamma}D_\beta(T)-\frac{\delta R^2 T}{2}+\frac{ TR^2}{2\gamma}-\frac{1}{2}(\mu-\delta)\lVert u_1-x_{1}\rVert^2$ can be zero. 
According to the prior setting $\delta\gamma<1$ and the chosen $\gamma$ here, by choosing a small $\delta$, $\frac{2RT^{-\beta} D_\beta(T)+R^2}{R^2+(\frac{\mu}{\delta}-1)\frac{1}{T}\lVert u_1-x_{1}\rVert^2}=\delta\gamma$ can be smaller than 1. 
Invoking the $\gamma$, together with (\ref{le9}), yeilds that
\begin{align}
	&\sum_{t=1}^T(f_t(x_t)+r_t(x_t)-f_t(u_t)-r_t(u_t))\nonumber\\
	\leqslant&\frac{M^2}{2\delta R}(1+\log T)(2T^{-\beta} D_\beta(T)+R).
\end{align}

Therefore, one has
\begin{align}
	\bm{Reg}_T^d=\mathcal{O}(\log T(1+T^{-\beta}D_\beta(T))).
\end{align}
%
\hfill\rule{2mm}{2mm}

\bibliographystyle{IEEEtran}
\bibliography{ref}

\vfill

\end{document}